\documentclass[11pt,reqno]{amsart}
\usepackage{fullpage}
\usepackage{amsmath,amsthm,verbatim,amssymb,amsfonts,amscd, graphicx,bbm,bm}
\usepackage{graphics}
\usepackage{mathtools}
\usepackage{mathrsfs}
\usepackage{eufrak}
\usepackage{esint}
\usepackage{color}
\usepackage{tikz,tikz-cd}

\usepackage[pdfstartview=FitH, bookmarksnumbered=true,bookmarksopen=true, colorlinks=true, pdfborder=001, citecolor=blue, linkcolor=blue,urlcolor=blue]{hyperref}

\newtheorem{theorem}{Theorem}
\newtheorem{proposition}[theorem]{Proposition}
\newtheorem{lemma}[theorem]{Lemma}

\theoremstyle{definition}
\newtheorem{remark}{Remark}

\newtheorem{definition}[theorem]{Definition}

\newcommand{\cref}[1]{Corollary~\ref{c.#1}}

\numberwithin{equation}{section}
\numberwithin{theorem}{section}

\newcommand{\eps}{\varepsilon}

\title[Homogenization of eigenvalues for problems with high-contrast inclusions]{Homogenization of eigenvalues for problems with high-contrast inclusions}
\author{Xin Fu}
\address[X. Fu]{Yau Mathematical Sciences Center, Tsinghua University, Beijing 100084, P.R. China}
\email{fux20@mails.tsinghua.edu.cn}

\date{\today}

\begin{document}
	\maketitle	
\begin{abstract}
	We study quantitative homogenization of the eigenvalues for elliptic systems with periodically distributed inclusions, where the conductivity of inclusions are strongly contrast to that of the matrix. We propose a quantitative version of periodic unfolding method, based on this and the recent results concerned on high-contrast homogenization, the convergence rates of eigenvalues are studied for any contrast $\delta \in (0,\infty)$. 

	
	
	\noindent{\bf Mathematics subject classification (MSC 2020)}: 35B27, 35J70, 35P20
\end{abstract}

\section{Introduction}

In this paper, we consider homogenization for eigenvalues of the operator $\mathcal{L}_{\varepsilon, \delta}$ which is described as following: Let $\Omega$ be a bounded Lipschitz domain in $\mathbb{R}^d$, we define the unbounded operator $\mathcal{L}_{\varepsilon,\delta} $ on $[L^2(\Omega)]^m$ by
\begin{equation}
	\mathcal{L}_{\varepsilon,\delta} := -\frac{\partial }{\partial x_i }\left[ \Lambda_{\delta}^{\varepsilon} (x)a_{ij}^{\alpha \beta}  \left( \frac{x}{\varepsilon }\right) \frac{\partial}{\partial x_j}\right] = -\mathrm{div} \left[ \Lambda^{\varepsilon}_{\delta} (x) A\left( \frac{x}{\varepsilon}\right) \nabla \right]
\end{equation}
for $1\leq i,j \leq d, 1\leq \alpha ,\beta \leq m,  0< \varepsilon<1,  0<\delta <\infty$, with the domain
\begin{equation}
	\mathcal{D}( \mathcal{L}_{\varepsilon,\delta}): = \Big\{  u \in [ H_0^1(\Omega) ]^m : \mathcal{L}_{\varepsilon,\delta}\,u \in [ L^2(\Omega) ]^m \Big\}. 
\end{equation}

The coefficient tensor $A(y) = \big( a_{ij}^{\alpha \beta}(y)\big) \in H^1(Y)$ is assumed to be real and satisfy
\begin{itemize}
	\item[(a)] (Ellipticity) There exists $\nu>0$ such that 
	\begin{equation}\label{ellipticc}
		\nu |\xi|^2 \leq a_{ij}^{\alpha \beta}(y )\xi_i^{\alpha}\xi_j^{\beta} \leq \frac{1}{\nu} |\xi|^2 \quad \mathrm{for} \ y\in \mathbb{R}^d \ \mathrm{and} \ \xi = (\xi_i^{\alpha}) \in \mathbb{R}^{dm}.
	\end{equation} 
	\item[(b)] (Periodicity) 
	\begin{equation}\label{periodicc}
		A(y+\mathbf{n}) = A(y) \quad \mathrm{for}\ y\in \mathbb{R}^d \ \mathrm{and} \ \mathbf{n}\in \mathbb{Z}^d.
	\end{equation} 
	\item[(c)] (H\"{o}lder continuity)  For any $y,w \in \mathbb{R}^d$,
	\begin{equation}\label{smoothc}
		|A(y)-A(w)|\leq \tau |y-w|^{\lambda} \quad \mathrm{for}\ \mathrm{some } \ \lambda \in (0,1)\ \mathrm{and\ } \tau>0.
	\end{equation}
	\item[(d)] (Symmetry) For any $y \in \mathbb{R}^d$,
	\begin{equation}\label{symmeticc}
		a_{ij}^{\alpha \beta} (y) = a_{ji}^{\beta \alpha}(y) \quad \mathrm{for} \ 1\leq i,j \leq d \ \mathrm{and}\ 1\leq \alpha ,\beta \leq m.
	\end{equation}
\end{itemize}

The scalar function $\Lambda_{\delta}^{\varepsilon} (x): = \delta  \mathbbm{1}_{D_{\varepsilon}} (x)+ \mathbbm{1}_{\Omega_{\varepsilon}}(x)$ models the contrast between inclusions and matrix: Let $Y=[0,1)^d$ be the unit cell,
and let $\omega \subset Y$ be an open subset with connected Lipschitz boundary such that $\mathrm{dist}(\omega,\partial Y)>0$; for simplicity, assume $\omega$ is simply connected. $\omega$ is then the model inclusion in the unit scale. Given $\varepsilon>0$ and $\mathbf{n}\in \mathbb{Z}^d$, we denote $\eps(\mathbf{n}+Y)$ and $\eps(\mathbf{n}+\omega)$ by $Y_{\varepsilon}^{\mathbf{n}}$  and $\omega_{\varepsilon}^{\mathbf{n}}$, respectively. 
Let $\Pi_{\varepsilon}$ be the set of lattice points $\mathbf{n}$ such that $\overline{Y_{\varepsilon}^{\mathbf{n}} }$ be contained in $\Omega$, i.e.,
\begin{equation}
	\Pi_{\varepsilon}:=\left\{   \mathbf{n}\in \mathbb{Z}^d:  \overline{Y_{\varepsilon}^{\mathbf{n}} }\subset  \Omega  \right\}.
\end{equation}
Then the inclusions set $D_\eps$ and the matrix part $\Omega_\eps$ are defined by
\begin{equation}
	\label{eq:Depsdef}
	D_{\varepsilon}:=\bigcup_{\mathbf{n}\in \Pi_{\varepsilon} } \omega_{\varepsilon}^{\mathbf{n}} , \qquad \Omega_{\varepsilon}:=\Omega \setminus \overline{D_{\varepsilon}}.
\end{equation}

Under these conditions, we study the quantitative asymptotic behavior of the eigenvalues of $\mathcal{L}_{\varepsilon, \delta}$ as $\varepsilon \rightarrow 0$. To state the main results, let $\widehat{\mathcal{L}}_{\delta}$ be the homogenized operator defined on $[ L^2(\Omega) ]^m$ by
\begin{equation}
\widehat{\mathcal{L}}_{\delta}: =-\frac{\partial }{\partial x_i }\left[ \widehat{a}_{ij,\delta}^{\alpha \beta}   \frac{\partial}{\partial x_j}\right] =  - \mathrm{div}\,( \widehat{A}_{\delta} \nabla ),
\end{equation}
with the domain
\begin{equation}
	\mathcal{D}(\widehat{\mathcal{L}}_{\delta}) : = \Big\{  u \in [ H_0^1(\Omega) ]^m : \widehat{\mathcal{L}}_{\delta}\,u \in [ L^2(\Omega) ]^m \Big\},
\end{equation}
where the coefficient tensor
$\widehat{A}_{\delta} = \big(\widehat{a}_{ij,\delta}^{\alpha \beta}\big)$ is defined by
\begin{equation}
	\widehat{a}_{ij,\delta}^{\alpha \beta} = \int_Y \Lambda_{\delta}(y) \Big[ a_{ij}^{\alpha \beta}(y) + a_{ik}^{\alpha \gamma}(y)   \frac{\partial}{\partial y_k} \chi_{j,\delta}^{\gamma \beta}\Big]\,dy. 
\end{equation}
Here for each $i$ and $\alpha$, $\chi_{i,\delta}^{\alpha }= \big(\chi_{i,\delta}^{\alpha 1}, \cdots, \chi_{i,\delta}^{\alpha m}\big)$ is the solution of the \textit{cell problem}
\begin{equation}
	\label{eq:cellproblem}
	\left\{
	\begin{aligned}
		& \mathcal{L}_{1,\delta} \big(\chi_{i,\delta}^{\alpha} + y_ie^{\alpha}\big)  =0 \qquad \mathrm{in}\ Y, \\
		& \chi_{i,\delta}^{\alpha} \ \mathrm{is \ } Y \mathrm{- periodic\ and \ mean \ zero}, 
	\end{aligned}
	\right.
\end{equation}
where $e^{\alpha} = (0,\cdots, 1 ,\cdots, 0) \in \mathbb{R}^m  $ with $1$ in the $\alpha^{\mathrm{th}}$ position. We denote the tensor $\chi_{\delta}  = \big(\chi_{i,\delta}^{\alpha \beta} \big)$.

Let 
\begin{equation}
	\widehat{\Pi}_{\varepsilon}:=\left\{   \mathbf{n}\in \mathbb{Z}^d:  \overline{Y_{\varepsilon}^{\mathbf{n}} }\cap  \Omega \neq \emptyset  \right\}, \quad \mathrm{and} \quad 	\widehat{\Omega}_{\varepsilon} := \bigcup_{\mathbf{n}\in \widehat{\Pi}_{\varepsilon} } \overline{ Y^{\mathbf{n}}_{\varepsilon}}.
\end{equation}

\begin{theorem}\label{theorem 1}
	Let $\lambda^i_{\varepsilon,\delta}$ be the $i$-th eigenvalue of $\mathcal{L}_{\varepsilon,\delta}^{-1}  $ in the decreasing order, and let $\eta^i_{\varepsilon,\delta}$ be the $i$-th eigenvalue of
	\begin{equation}
	\widehat{\mathcal{L}}_{\delta}^{-1} \mathbbm{1}_{\Omega} \langle \cdot \rangle_Y + \kappa \mathcal{P}_{\varepsilon} \mathcal{L}_{\omega,y}^{-1}:[L^2(\widehat{\Omega}_{\varepsilon} \times Y)]^m \rightarrow [ L^2(\widehat{\Omega}_{\varepsilon}  \times Y)  ]^m
	\end{equation} 
 in the decreasing order, then there exists a constant $C>0$, depends only on $\Omega$ and $\omega$, such that
	\begin{equation}\label{eigenrate1}
		\big|\lambda^i_{\varepsilon,\delta}- \eta^i_{\varepsilon,\delta} \big| \leq C\varepsilon^{\frac{1}{2}},
	\end{equation}
where $\langle \cdot \rangle_Y$ is the integral operator with respect to $y$ variable defined in Proposition \ref{quantitative_estimates}, $\kappa = \delta^{-1} \varepsilon^2$, $\mathcal{P}_{\varepsilon}$ is the projection operator in $[ L^2(\widehat{\Omega}_{\varepsilon}  \times Y) ]^m$ on piecewise constant functions in $x$ in each cell $Y^{\mathbf{n}}_{\varepsilon}$ defined in Proposition \ref{quanti_unfolding_prop}, and $\mathcal{L}_{\omega,y}^{-1}\, f$ is the solution of 
\begin{equation}\left\{
	\begin{aligned}
		& \mathcal{L}_{1,1} \,u = f & \mathrm{in} \ \omega, \\
		& u = 0 & \mathrm{on} \ Y \setminus \omega.
	\end{aligned}\right.
\end{equation}
\end{theorem}

 It is clear that $\widehat{\mathcal{L}}_{\delta}^{-1} \mathbbm{1}_{\Omega} \langle \cdot \rangle_Y + \kappa \mathcal{P}_{\varepsilon} \mathcal{L}_{\omega,y}^{-1}$ is a positive compact self-adjoint operator on $[ L^2(\widehat{\Omega}_{\varepsilon}\times Y) ]^m$, this yields that its spectrum is discrete.
Given $\varepsilon >0$, we decompose the spectrum into  two parts:
\begin{itemize}
	\item[(a)] The first part comprises all eigenvalues for which the corresponding eigenfunction $u$ has a zero mean in $Y$, i.e. $\int_Y u(x,y)\,dy =0 $. This part is referred to as the pure Bloch spectrum.
	\item[(b)] The second part consists of all eigenvalues such that the corresponding eigenfunction $u$ has a nonzero mean in $Y$, i.e. $\int_Y u(x,y)\,dy \neq 0 $. We call this part as the residual spectrum.
\end{itemize}

The next two theorems characterize the pure Bloch spectrum and the residual spectrum, respectively.

\begin{theorem}\label{thm:bloch spectrum}
	The pure Bloch spectrum is ordered by
	\begin{equation}\label{sigma1multi}
		\underbrace{\kappa \alpha_1=\cdots=\kappa \alpha_1}_{|\widehat{\Pi}_{\varepsilon}| \ \mathrm{terms}} \geq \cdots \geq  \underbrace{\kappa \alpha_i=\cdots=\kappa \alpha_i}_{|\widehat{\Pi}_{\varepsilon}| \ \mathrm{terms}} \geq \cdots \searrow 0,
	\end{equation}
	where the sequence $(\alpha_i)_{i \geq 1}$, ordering decreasingly, denotes the  eigenvalues of $\mathcal{L}_{\omega,y}^{-1}$ with associated mean-zero eigenfunction.
\end{theorem}

Let 
\begin{equation*}
	\gamma_{\kappa}(\lambda) = -  \int_Y (\kappa  \mathcal{L}_{\omega,y}^{-1} - \lambda   )^{-1} [I_m](y) \,dy,
\end{equation*}
then we observe that:
\begin{itemize}
	\item[(a)] If $\delta >>\varepsilon^2$, then $\lim_{\varepsilon\rightarrow 0} \kappa =0$, and
	\begin{equation*}
		\lim_{\varepsilon \rightarrow 0} \gamma_{\kappa} (\lambda) = \lambda^{-1} I_m.
	\end{equation*}
	\item[(b)] If $\delta <<\varepsilon^2$, then $\lim_{\varepsilon\rightarrow 0} \kappa =\infty$, and
	\begin{equation*}
		\lim_{\varepsilon \rightarrow 0} \gamma_{\kappa} (\lambda) = \lambda^{-1} (1-\theta)I_m,
	\end{equation*}
	where $\theta =|\omega|$ is the Lebesgue measure of the unit inclusion $\omega$.
	\item[(c)] In the critical case where $\delta \approx \varepsilon^2$, i.e., $\kappa = O(1)$, $\gamma_{\kappa} (\lambda)$ is a nontrivial symmetric matrix. 
\end{itemize}
The residual spectrum can be depicted using $\gamma_{\kappa}$. Given $\varepsilon>0$, we denote the residual spectrum by $\mathrm{RSpec}_{\varepsilon}$. 

\begin{theorem}\label{thm:residual spectrum}
	For any $\gamma_{\kappa}(\lambda) \in \mathrm{Spec}\, \widehat{\mathcal{L}}_{\delta} $ with multiplicity $k$, there exists some $\lambda_{\varepsilon} \in  \mathrm{RSpec}_{\varepsilon}$ with multiplicity $k$, such that 
	\begin{equation}
		\| \gamma_{\kappa}(\lambda_{\varepsilon} )^{-1} - \gamma_{\kappa}(\lambda)^{-1} \| \leq C\varepsilon \|  I_m - \lambda_{\varepsilon}^{-1 }   \gamma_{\kappa} (\lambda_{\varepsilon})^{-1} \|  .
	\end{equation}
Moreover, 
	we have
\begin{equation}
	\gamma_{\kappa}(\mathrm{RSpec}_{\varepsilon})^{-1} \subset  \bigcup_{\gamma_{\kappa}(\lambda) \in \mathrm{Spec}\, \widehat{\mathcal{L}}_{\delta}   }   B\Big( \gamma_{\kappa} (\lambda)^{-1}, C\varepsilon \|  I_m - \lambda_{\varepsilon}^{-1 }   \gamma_{\kappa} (\lambda_{\varepsilon})^{-1} \| \Big).
\end{equation}
\end{theorem}

\begin{remark}
	As an immediately corollary of Theorem \ref{theorem 1}, Theorem \ref{thm:bloch spectrum} and Theorem \ref{thm:residual spectrum}, we obtain the qualitative result:
	\begin{equation}
		\lim_{\varepsilon \rightarrow 0} \mathrm{Spec} \, \mathcal{L}_{\varepsilon,\delta} = \left\{
		\begin{aligned}
			& \mathrm{Spec}\, \widehat{\mathcal{L}}_{\delta} & \mathrm{for}\ \delta >>\varepsilon^2, \\
			& \Big( \lim_{\varepsilon \rightarrow 0} \delta \varepsilon^{-2}  \{ \alpha_m^{-1}\}_{m \geq 1}  \Big) \cup (1-\theta)^{-1} \mathrm{Spec}\, \widehat{\mathcal{L}}_0 & \mathrm{for}\ \delta <<\varepsilon^2, \\
			& \kappa^{-1}  \{ \alpha_m^{-1}\}_{m \geq 1} \cup \overline{ \{ \lambda^{-1} : \gamma_{\kappa}(\lambda) \in \mathrm{Spec}\, \widehat{\mathcal{L}}_0\} } & \mathrm{for}\ \delta \approx \varepsilon^2.
		\end{aligned}\right.
	\end{equation}
\end{remark}
\begin{remark}
	More careful analysis of the convergence rate are presented in Section \ref{sec:scalar case} in the scalar case $m=1$, see Theorem \ref{thm:scalar}.
\end{remark}

This paper is organized as following: In Section \ref{sec: periodic unfolding method}, we introduce the quantitative periodic unfolding method. In Section \ref{sec:homo eigenvalue}, we prove Theorem \ref{theorem 1}, Theorem \ref{thm:bloch spectrum} and Theorem \ref{thm:residual spectrum}. In Section \ref{sec:scalar case}, we show a more delicate analysis in the scalar case $m=1$. These completes this paper.

For the sake of clarity and simplicity, we will proceed by omitting the upper symbol $m$ from our notation. Instead of the product space $H^m$, we will use $H$. Here, $H$ denotes a variety of function spaces, including but not limited to $L^2(\Omega)$, $H^1(\Omega)$, and $L^2(\Omega \times Y)$.

\section{Quantitative periodic unfolding method}\label{sec: periodic unfolding method}

In this section, we introduce the quantitative periodic unfolding method. In 1990, Arbogast, Douglas and Hornung \cite{doi:10.1137/0521046} introduced a ‘dilation’ operation to study homogenization in a periodic medium with double porosity. This technique reduces two-scale convergence to weak convergence in an appropriate space. Combining this approach with ideas from Finite Element approximations, Cioranescu, Damlamian and Griso \cite{MR1921004} propose the periodic unfolding method to study homogenization of multiscale periodic problems. For further details, we refer to \cite{allaire_bloch_1998, cioranescu_periodic_2018}. 

The most significant advantage of the periodic unfolding method is that it lifts the notion of (weak) two-scale convergence in $L^2(\Omega)$ to the (weak) convergence in the unfolded space $L^2(\Omega \times Y)$. Nevertheless, to study the convergence rates in problems of homogenization, quantitative estimates for the corresponding operators are required. These estimates are provided in this section.

\begin{definition}
	For $x \in \mathbb{R}^d$, let $[x]_Y\in \mathbb{Z}^d$ be the integer part of $x$, and $\{x\}_Y=x -[x]_Y \in Y$ be the fractional part of $x$. 
	\begin{enumerate}
		\item [(a)] The unfolding operator $\widetilde{\mathcal{T}}_{\varepsilon}:L^2(\widehat{\Omega}_{\varepsilon} ) \rightarrow L^2(\widehat{\Omega}_{\varepsilon}  \times Y)$ is defined by
		\begin{equation*}
			\widetilde{\mathcal{T}}_{\varepsilon} u(x,y)=\sum_{\mathbf{n}\in \widehat{\Pi}_{\varepsilon}} \mathbbm{1}_{Y^{\mathbf{n}}_{\varepsilon}} (x)u\left( \varepsilon  \left[ \frac{x}{\varepsilon}\right]_{Y}+\varepsilon y \right) .
		\end{equation*}
		\item [(b)] The local averaging operator $\widetilde{\mathcal{U}}_{\varepsilon} :L^2(\widehat{\Omega}_{\varepsilon}  \times Y) \rightarrow L^2(\widehat{\Omega}_{\varepsilon} )$ is defined by
		\begin{equation*}
			\widetilde{\mathcal{U}}_{\varepsilon} \phi (x) = \sum_{\mathbf{n}\in \widehat{\Pi}_{\varepsilon}} \mathbbm{1}_{Y^{\mathbf{n}}_{\varepsilon}} (x)\int_Y \phi \left( \varepsilon  \left[ \frac{x}{\varepsilon}\right]_Y+\varepsilon z,  \left\{ \frac{x}{\varepsilon}\right\} _Y \right)\,dz.
		\end{equation*}
	\end{enumerate}
\end{definition}

The following proposition states the basic properties of $\widetilde{\mathcal{T}}_{\varepsilon}$ and $\widetilde{\mathcal{U}}_{\varepsilon}$, whose proof could be found in \cite{allaire_bloch_1998}.

\begin{proposition}\label{quanti_unfolding_prop}
	Let $\Omega \subset \mathbb{R}^d$ be a bounded Lipschitz domain.
	\begin{itemize}
		\item[(a)] $\widetilde{\mathcal{T}}_{\varepsilon}$ and $\widetilde{\mathcal{U}}_{\varepsilon}$ are both bounded by norm $1$, and $\widetilde{\mathcal{T}}_{\varepsilon}$ is the adjoint of $\widetilde{\mathcal{U}}_{\varepsilon}$, i.e., $\widetilde{\mathcal{U}}_{\varepsilon}^* = \widetilde{\mathcal{T}}_{\varepsilon}$. \item[(b)] $\widetilde{\mathcal{U}}_{\varepsilon}\circ \widetilde{\mathcal{T}}_{\varepsilon} = \mathrm{Id}_{L^2(\widehat{\Omega}_{\varepsilon})}$ is the identity operator on $L^2( \widehat{\Omega}_{\varepsilon})$.
		\item[(c)]  $\mathcal{P}_{\varepsilon} := \widetilde{\mathcal{T}}_{\varepsilon} \circ\widetilde{\mathcal{U}}_{\varepsilon}$ is a projection operator in $L^2(\widehat{\Omega}_{\varepsilon} \times Y)$ on piecewise constant function in $x$ in each cell $Y^{\mathbf{n}}_{\varepsilon}$ for $\mathbf{n}\in \widehat{\Pi}_{\varepsilon}$. More precisely, for any $\phi \in L^2(\widehat{\Omega}_{\varepsilon} \times Y)$,
		\begin{equation}\label{compo_projection}
			\mathcal{P}_{\varepsilon} \phi(x,y)  = \sum_{\mathbf{n}\in \widehat{\Pi}_{\varepsilon}} \mathbbm{1}_{Y^{\mathbf{n}}_{\varepsilon}} (x) \varepsilon^{-d} \int_{Y^{\mathbf{n}}_{\varepsilon}} \phi(x',y)\,dx'.
		\end{equation}
	\end{itemize}
\end{proposition}

We denote that
\begin{equation}
	\mathcal{T}_{\varepsilon} = \widetilde{\mathcal{T}}_{\varepsilon} \circ \mathbbm{1}_{\Omega}, \quad \mathrm{and}\quad  \mathcal{U}_{\varepsilon} = \mathbbm{1}_{\Omega} \circ  \widetilde{\mathcal{U}}_{\varepsilon} ,
\end{equation}
then $\mathcal{T}_{\varepsilon}  $ is the adjoint of $\mathcal{U}_{\varepsilon}  $, and $\mathcal{U}_{\varepsilon}  \circ \mathcal{T}_{\varepsilon}  = \mathbbm{1}_{\Omega}$.

As we have mentioned above, the periodic unfolding method transforms the concept of two-scale convergence in $L^2(\Omega)$ into the typical convergence in $L^2(\Omega \times Y)$. Hence, the two-scale compactness becomes the usual compactness in $L^2(\Omega \times Y)$. This facilitates an easy transition from the equation under consideration to the homogenized equation at the homogenization limit. However, to obtain the rate of convergence, one must consider the asymptotic limit of the unfolding operators. This is presented in Proposition \ref{quantitative_estimates}. Before that, we need a lemma:

\begin{lemma}\label{lem_difference_gradient}
	Let $D\subset \mathbb{R}^d$ be a bounded Lipschitz domain, there exists a constant $C>0$, depends only on $D$,  such that for any $u \in H^1(D)$, 
	\begin{equation*}
		\int_D\int_D |u(x)-u(y)|^2\,dxdy \leq C \| \nabla u\|^2_{L^2(D)}.
	\end{equation*}
\end{lemma}
\begin{proof}
	This is just a corollary of the Poincar\'{e}-Wirtinger inequality. We note that
	\begin{equation*}
		|u(x)-u(y)|^2 \leq 2|u(x) - \langle u \rangle_D |^2 + 2|u(y) -  \langle u \rangle_D |^2 ,
	\end{equation*}
	where $\langle u\rangle_D = \int_D u(x)\,dx$. Therefore, 
	\begin{equation*}
			\int_D\int_D |u(x)-u(y)|^2\,dxdy \leq  2 |D| \Big(\int_D |u(x) - \langle u \rangle_D |^2 \,dx + \int_D |u(y) - \langle u \rangle_D |^2 \,dy \Big) \leq C \| \nabla u\|^2_{L^2(D)},
	\end{equation*}
	where the second inequality using the Poincar\'{e}-Wirtinger inequality.
\end{proof}

\begin{proposition}[Quantitative estimates]\label{quantitative_estimates}
	Let $\langle \cdot \rangle_Y:L^2(\widehat{\Omega}_{\varepsilon}  \times Y) \rightarrow L^2(\widehat{\Omega}_{\varepsilon} )$ be the integral operator with respect to $y$ variable, i.e.,
	\begin{equation*}
		\langle \cdot \rangle_Y : \phi(x,y) \mapsto \int_Y \phi (x,y)\,dy.
	\end{equation*}
	Let $\iota:L^2(\widehat{\Omega}_{\varepsilon} ) \rightarrow L^2(\widehat{\Omega}_{\varepsilon}  \times Y)$ be the embedding operator, i.e.,
	\begin{equation*}
		\iota: u(x) \rightarrow u(x,y).
	\end{equation*} 
	It is clear that $\iota^* = \langle \cdot \rangle_Y$. Then there exists a constant $C>0$, depends only on $\Omega$, such that
	\begin{itemize}
		\item[(a)]   $	\| \widetilde{\mathcal{U}}_{\varepsilon} - \langle \cdot \rangle_Y \|_{L^2(\widehat{\Omega}_{\varepsilon}  \times Y) \rightarrow H^{-1}(\widehat{\Omega}_{\varepsilon} )}\leq C\varepsilon$.
		\item[(b)] $\| \widetilde{\mathcal{T}}_{\varepsilon} - \iota \|_{H^1(\widehat{\Omega}_{\varepsilon}  ) \rightarrow  L^2(\widehat{\Omega}_{\varepsilon}  \times Y)}\leq C \varepsilon$.
		\item[(c)] $\| \mathcal{P}_{\varepsilon} - \mathrm{Id}_{L^2(\widehat{\Omega}_{\varepsilon} )} \|_{H^1(\widehat{\Omega}_{\varepsilon} ) \rightarrow L^2(\widehat{\Omega}_{\varepsilon}  )}\leq C\varepsilon$. 
	\end{itemize}
\end{proposition}

\begin{proof} We prove the proposition one item by one item.
	\begin{itemize}
		\item[(a)] 
		We verify that $	\| \widetilde{\mathcal{U}}_{\varepsilon} - \langle \cdot \rangle_Y \|_{L^2(\widehat{\Omega}_{\varepsilon}  \times Y) \rightarrow H^{-1}(\widehat{\Omega}_{\varepsilon} )}\leq C\varepsilon$. Given $v \in L^2(\widehat{\Omega}_{\varepsilon} \times Y)$ and $\phi \in H_0^1(\widehat{\Omega}_{\varepsilon})$, by definition we have
		\begin{equation*}
				 \big\langle  \widetilde{\mathcal{U}}_{\varepsilon} v -\langle v \rangle_Y ,\phi \big\rangle_{L^2(\widehat{\Omega}_{\varepsilon})}  =  \sum_{\mathbf{n} \in \widehat{\Pi}_{\varepsilon }} \int_{Y^{\mathbf{n}}_{\varepsilon }}\Big( \int_Y v \left( \varepsilon  \left[ \frac{x}{\varepsilon}\right]_Y+\varepsilon z,  \left\{ \frac{x}{\varepsilon}\right\} _Y \right)\,dz-\int_Y v(x,w)\,dw \Big)\phi  (x)\,dx ,
		\end{equation*}
	using change of variable $x=\varepsilon \mathbf{n}+\varepsilon y$, we get
		\begin{equation*}
			\begin{aligned}
					\big\langle  \widetilde{\mathcal{U}}_{\varepsilon} v -\langle v \rangle_Y ,\phi \big\rangle_{L^2(\widehat{\Omega}_{\varepsilon})}  = \varepsilon^d    \sum_{\mathbf{n} \in \widehat{\Pi}_{\varepsilon }} \int_Y V^{\mathbf{n}}_{\varepsilon}(y) \phi  (\varepsilon \mathbf{n}+\varepsilon y) \,dy ,
			\end{aligned}
		\end{equation*}
		where
		\begin{equation*}
			V^{\mathbf{n}}_{\varepsilon}(y):=\int_Y v(\varepsilon {\mathbf{n}}+\varepsilon z,y)\,dz-\int_Y v(\varepsilon {\mathbf{n}}+\varepsilon y ,w)\,dw .
		\end{equation*}
		Since $\int_Y V^{\mathbf{n}}_{\varepsilon}(y)\,dy= 0$, we obtain that 
		\begin{equation*}
			\begin{aligned}
			\big| \big\langle  \widetilde{\mathcal{U}}_{\varepsilon} v -\langle v \rangle_Y ,\phi \big\rangle_{L^2(\widehat{\Omega}_{\varepsilon})}   \big| 
				& \leq \varepsilon^d   \left| \sum_{\mathbf{n} \in \widehat{\Pi}_{\varepsilon }} \int_Y V^{\mathbf{n} }_{\varepsilon }(y)\Big(\phi  (\varepsilon \mathbf{n} +\varepsilon y)  - \int_Y \phi  (\varepsilon \mathbf{n} +\varepsilon s) \,ds\Big)\,dy \right|  \\
				& \leq  \varepsilon^d   \sum_{\mathbf{n} \in \widehat{\Pi}_{\varepsilon }}  \| V^{\mathbf{n}}_{\varepsilon }\|_{L^2(Y)} \left\| \phi  (\varepsilon \mathbf{n} +\varepsilon y)  - \int_Y \phi  (\varepsilon \mathbf{n} +\varepsilon s) \,ds\right\|_{L^2(Y)} \\
				& \leq C\varepsilon^{\frac{d}{2}+1}  \left( \sum_{\mathbf{n} \in \widehat{\Pi}_{\varepsilon }} \int_Y |V^{\mathbf{n}}_{\varepsilon}|^2 \right)^{1/2}  \| \phi \|_{H^1(\widehat{\Omega}_{\varepsilon})},
			\end{aligned}
		\end{equation*}
		where the last inequality follows from the Poincar\'{e}-Wirtinger inequality on $Y$. Now by definition, we have
		\begin{equation*}
			\begin{aligned}
				|V^{\mathbf{n}}_{\varepsilon}(y)|^2 \leq 2\int_Y |v(\varepsilon \mathbf{n}+\varepsilon z,y)|^2\,dz + 2\int_Y |v(\varepsilon \mathbf{n} +\varepsilon y,w)|^2\,dw,
			\end{aligned}
		\end{equation*}
		so
		\begin{equation*}
			\begin{aligned}
				\sum_{\mathbf{n} \in \widehat{\Pi}_{\varepsilon }} \int_Y |V^{\mathbf{n}}_{\varepsilon}(y)|^2\,dy &\leq  2\sum_{\mathbf{n} \in \widehat{\Pi}_{\varepsilon }} \int_Y \int_Y   |v(\varepsilon \mathbf{n}+\varepsilon z,y)|^2\,dz dy  \\
				& =2\varepsilon^{-d}  \sum_{\mathbf{n} \in \widehat{\Pi}_{\varepsilon }} \int_Y \int_{Y^{\mathbf{n}}_{\varepsilon}}   |v(x,y)|^2\,dx dy \\
				&= 2\varepsilon^{-d} \int_{\widehat{\Omega}_{\varepsilon} \times Y} |v(x,y)|^2\,dxdy.
			\end{aligned}
		\end{equation*}
		This yields the desired conclusion.
		\item[(b)] 
		We verify that $\| \widetilde{\mathcal{T}}_{\varepsilon} - \iota \|_{H^1(\widehat{\Omega}_{\varepsilon}  ) \rightarrow  L^2(\widehat{\Omega}_{\varepsilon}  \times Y)}\leq C \varepsilon$. Given $u \in H^1(\widehat{\Omega}_{\varepsilon})$, we have
		\begin{equation*}
				\int_{\widehat{\Omega}_{\varepsilon} \times Y} | \widetilde{\mathcal{T}}_{\varepsilon} u(x,y)-u(x) |^2\,dxdy\\
				= \varepsilon^d \sum_{\mathbf{n}\in \widehat{\Pi}_{\varepsilon}} \int_{Y \times Y}  \left|u( \varepsilon  \mathbf{n}+\varepsilon y ) - u(\varepsilon \mathbf{n}+ \varepsilon s)\right|^2\,dsdy ,
		\end{equation*}
		by Lemma \ref{lem_difference_gradient} for $D =Y$, we obtain that
		\begin{equation*}
			\begin{aligned}
				\int_{\widehat{\Omega}_{\varepsilon} \times Y} | \widetilde{\mathcal{T}}_{\varepsilon} u(x,y)-u(x) |^2\,dxdy&\leq C\varepsilon^{d+2} \sum_{\mathbf{n}\in \widehat{\Pi}_{\varepsilon}} \int_Y |\nabla u(\varepsilon\mathbf{n}+\varepsilon x)|^2 \,dx  \\
				& \leq C\varepsilon^2 \|  u\|_{H^1(\widehat{\Omega}_{\varepsilon})}^2,
			\end{aligned}
		\end{equation*}
	which is the desired conclusion.
		\item[(c)] 
		Lastly we verify that $\| \mathcal{P}_{\varepsilon} - \mathrm{Id}_{L^2(\widehat{\Omega}_{\varepsilon} )} \|_{H^1(\widehat{\Omega}_{\varepsilon} ) \rightarrow L^2(\widehat{\Omega}_{\varepsilon}  )}\leq C\varepsilon$. Given $u \in H^1(\widehat{\Omega}_{\varepsilon} )$, we compute 
		\begin{equation*}
				\| \mathcal{P}_{\varepsilon} u - u \|^2_{L^2(\widehat{\Omega}_{\varepsilon})} = \sum_{\mathbf{n} \in \widehat{\Pi}_{\varepsilon}}   \int_{Y^{\mathbf{n}}_{\varepsilon}}\Big| \varepsilon^{-d} \int_{Y^{\mathbf{n}}_{\varepsilon}} u(x')\,dx' - u(x) \Big|^2\,dx  \leq C\varepsilon^2  \| u\|_{H^1(\widehat{\Omega}_{\varepsilon})}^2,
		\end{equation*}
		where the last inequality follows from Poincar\'{e}-Wirtinger inequality on $Y$.
	\end{itemize}
\end{proof}

\section{Homogenization of the eigenvalues}\label{sec:homo eigenvalue}

This section is devoted to provide proofs for Theorem \ref{theorem 1}, Theorem \ref{thm:bloch spectrum} and Theorem \ref{thm:residual spectrum}. Here is an overviw of our method:

\textit{Step 1}. Fu and Jing proved in their work \cite{fujing_homogenization} that

\begin{lemma}\label{lemma Fu and Jing}
	There exists a constant $C>0$, depends only on $d, m, \mu,\lambda,\tau, \kappa, \Omega$ and $\omega$, such that
	\begin{equation}
		\Big\|  \mathcal{L}^{-1}_{\varepsilon,\delta} -  \widehat{\mathcal{L}}^{-1}_{\delta} - \delta^{-1} \mathcal{L}_{D_{\varepsilon}}^{-1}    \Big\|_{L^2(\Omega) \rightarrow L^2(\Omega)} \leq C\varepsilon^{1/2}.
	\end{equation}
\end{lemma}

Lemma \ref{lemma Fu and Jing} suggests that the spectrum of $\mathcal{L}^{-1}_{\varepsilon,\delta}$ and 
$\widehat{\mathcal{L}}_{\delta}^{-1} + \delta^{-1} \mathcal{L}_{D_{\varepsilon}}^{-1}$ are nearly same. Recall that  $\mathcal{T}_{\varepsilon}$ is the adjoint of  $\mathcal{U}_{\varepsilon}$ and $\mathcal{U}_{\varepsilon}\circ \mathcal{T}_{\varepsilon} = \mathbbm{1}_{\Omega}$, the lifted (conjugated) operator $\mathcal{T}_{\varepsilon} \big( \widehat{\mathcal{L}}_{\delta}^{-1} + \delta^{-1} \mathcal{L}_{D_{\varepsilon}}^{-1} \big) \mathcal{U}_{\varepsilon}$ is self-adjoint and exhibits an identical spectrum to $\widehat{\mathcal{L}}_{\delta}^{-1} + \delta^{-1} \mathcal{L}_{D_{\varepsilon}}^{-1}$. These insights guide us towards studying the spectrum of the lifted operator.

We then utilize the quantitative periodic unfolding method introduced in Section \ref{sec: periodic unfolding method} to demonstrate that the lifted operator approaches $\widehat{\mathcal{L}}_{\delta}^{-1} \mathbbm{1}_{\Omega} \langle \cdot \rangle_Y + \kappa \mathcal{P}_{\varepsilon} \mathcal{L}_{\omega,y}^{-1}$ in operator norms with an error $O(\varepsilon)$. This validates the proof of Theorem \ref{theorem 1}.

\textit{Step 2}.  Proving Theorem \ref{thm:bloch spectrum} is straightforward. In order to prove Theorem \ref{thm:residual spectrum}, intuitively we infer that the spectrum of $\widehat{\mathcal{L}}_{\delta}^{-1}  \mathbbm{1}_{\Omega} \langle \cdot \rangle_Y + \kappa \mathcal{P}_{\varepsilon} \mathcal{L}_{\omega,y}^{-1}$ converges to that of $\widehat{\mathcal{L}}_{\delta}^{-1}   \mathbbm{1}_{\Omega} \langle \cdot \rangle_Y + \kappa \mathbbm{1}_{\Omega} \mathcal{L}_{\omega,y}^{-1}$ by Proposition \ref{quantitative_estimates} (c). Regrettably, the later (limit) operator is not compact, and the operator norm of their difference does not converges to zero. To circumvent this issue, we construct an auxillary operator $\widehat{\mathcal{L}}_{\delta}^{-1}  +B_{\varepsilon,\delta,\lambda} $ on $H^{-1}(\Omega)$ with a new inner product $\langle \cdot,\cdot \rangle_{\delta}$. By showing that $\| B_{\varepsilon,\delta,\lambda}  \|_{\delta} \rightarrow 0$, we successfully prove  Theorem \ref{thm:residual spectrum}.

This process is summarized in the following diagram:

\begin{tikzcd}
 & \mathcal{T}_{\varepsilon} \big( \widehat{\mathcal{L}}_{\delta}^{-1} + \delta^{-1} \mathcal{L}_{D_{\varepsilon}}^{-1} \big) \mathcal{U}_{\varepsilon} \arrow[r,"\varepsilon \rightarrow 0"]
	& \widehat{\mathcal{L}}_{\delta}^{-1}  \mathbbm{1}_{\Omega} \langle \cdot \rangle_Y + \kappa \mathcal{P}_{\varepsilon} \mathcal{L}_{\omega,y}^{-1} \arrow[r,"\varepsilon \rightarrow 0"] \arrow[d,"\mathrm{down}"]
	& \widehat{\mathcal{L}}_{\delta}^{-1} \mathbbm{1}_{\Omega} \langle \cdot \rangle_Y + \kappa  \mathbbm{1}_{\Omega} \mathcal{L}_{\omega,y}^{-1} \arrow[d,"\mathrm{down}"]  \\
	&\mathcal{L}_{\varepsilon,\delta}^{-1}   \approx 
	\widehat{\mathcal{L}}_{\delta}^{-1} + \delta^{-1} \mathcal{L}_{D_{\varepsilon}}^{-1}\arrow[u,"\mathrm{lift}"] & \widehat{\mathcal{L}}_{\delta}^{-1}  +B_{\varepsilon,\delta,\lambda} \arrow[r,"\varepsilon \rightarrow 0"] &  \widehat{\mathcal{L}}_{\delta}^{-1} 
\end{tikzcd}

We now give the proof of Theorem \ref{theorem 1}.

\noindent \textit{Proof of Theorem \ref{theorem 1}.}
Since $\mathcal{T}_{\varepsilon}$ is the adjoint of $\mathcal{U}_{\varepsilon}$ and $\mathcal{U}_{\varepsilon}\circ \mathcal{T}_{\varepsilon} = \mathbbm{1}_{\Omega}$, the spectrum of $\mathcal{L}_{\varepsilon,\delta}^{-1} $ coincides with the spectrum of $\mathcal{T}_{\varepsilon} \mathcal{L}_{\varepsilon,\delta}^{-1}  \mathcal{U}_{\varepsilon}$. By Lemma \ref{lemma Fu and Jing}, we obtain that 
	\begin{equation*}
		\big\| \mathcal{T}_{\varepsilon} \mathcal{L}_{\varepsilon,\delta}^{-1}  \mathcal{U}_{\varepsilon} - \mathcal{T}_{\varepsilon} \big( \widehat{\mathcal{L}}_{\delta}^{-1} + \delta^{-1} \mathcal{L}_{D_{\varepsilon}}^{-1} \big)  \mathcal{U}_{\varepsilon} \big\|_{L^2(\widehat{\Omega}_{\varepsilon} \times Y) \rightarrow L^2(\widehat{\Omega}_{\varepsilon}  \times Y)} \leq C\varepsilon^{1/2}.
	\end{equation*}
	We first observe that the commutate law $\varepsilon^{-2} \mathcal{L}_{D_{\varepsilon}}^{-1}  \mathcal{U}_{\varepsilon}= \mathcal{U}_{\varepsilon} \mathcal{L}_{\omega,y}^{-1} $ holds, so $\mathcal{T}_{\varepsilon} \big(  \delta^{-1} \mathcal{L}_{D_{\varepsilon}}^{-1} \big)  \mathcal{U}_{\varepsilon}= \kappa \mathcal{P}_{\varepsilon} \mathcal{L}_{\omega,y}^{-1}$. Moreover, by Proposition \ref{quantitative_estimates} (a) and (b), we have
	\begin{equation*}
		\begin{aligned}
		\| &  \mathcal{T}_{\varepsilon} \widehat{\mathcal{L}}_{\delta}^{-1} \mathcal{U}_{\varepsilon} - \widehat{\mathcal{L}}_{\delta}^{-1}  \mathbbm{1}_{\Omega} \langle \cdot \rangle_Y  \|_{L^2(\widehat{\Omega}_{\varepsilon}  \times Y \rightarrow L^2(\widehat{\Omega}_{\varepsilon}  \times Y))} \\
		& \leq 	\|   ( \widetilde{\mathcal{T}}_{\varepsilon} - \iota)\widehat{\mathcal{L}}_{\delta}^{-1} \mathcal{U}_{\varepsilon}   \|_{L^2(\widehat{\Omega}_{\varepsilon} \times Y) \rightarrow L^2(\widehat{\Omega}_{\varepsilon}  \times Y)} +  \| \widehat{\mathcal{L}}_{\delta}^{-1}\mathbbm{1}_{\Omega} ( \widetilde{\mathcal{U}}_{\varepsilon}  -  \langle \cdot \rangle_Y  ) \|_{L^2(\widehat{\Omega}_{\varepsilon}  \times Y) \rightarrow L^2(\Omega  )} \\
		& \leq C\varepsilon.
		\end{aligned}
	\end{equation*}
	These estimates, combine with the standard stability theorem for self-adjoint operators, yield the conclusion \eqref{eigenrate1}.
\hfill $\square$
\begin{remark}
	We note that the rate $O(\varepsilon^{1/2})$ follows from Lemma \ref{lemma Fu and Jing}, in fact, by exploring the regularity property of $\mathcal{L}_{\varepsilon,\delta}$, one may show the optimal $L^2$ convergence rate is $O(\varepsilon)$, as studied in \cite{shen_periodic_2018,shen_large-scale_2021}. We may study this aspect in a future work.
\end{remark}

\noindent \textit{Proof of Theorem \ref{thm:bloch spectrum}}.
	Assume that there exists a nonzero $\psi$ such that $\int_{\omega} \psi(y)\,dy=0$ and
	\begin{equation*}
		\mathcal{L}_{\omega,y}^{-1}\, \psi=\lambda  \psi,
	\end{equation*}
	we define $|\widehat{\Pi}_{\varepsilon}|$ numbers of independent functions $u^{\mathbf{n}} (x,y)= \mathbbm{1}_{Y^{\mathbf{n}}_{\varepsilon}}(x) \psi(y)$, where $\mathbf{n} \in \widehat{\Pi}_{\varepsilon}$, then
	\begin{equation*}
		\big( \widehat{\mathcal{L}}_{\delta}^{-1}  \mathbbm{1}_{\Omega} \langle \cdot \rangle_Y + \kappa \mathcal{P}_{\varepsilon} \mathcal{L}_{\omega,y}^{-1} \big) u^{\mathbf{n}} =  \kappa \mathcal{P}_{\varepsilon} \mathcal{L}_{\omega,y}^{-1}  u^{\mathbf{n}} =\kappa  \lambda u^{\mathbf{n}}  ,
	\end{equation*}
	which shows that \eqref{sigma1multi} is contained in the pure Bloch spectrum.
	
	Conversely, for any $\lambda $ in the pure Bloch spectrum, by definition, there exists an eigenfunction $u \in L^2(\widehat{\Omega}_{\varepsilon} \times Y)$ such that $\int_Y u(x,y)\,dy=0$ and 
	\begin{equation*}
			\big( \widehat{\mathcal{L}}_{\delta}^{-1} \mathbbm{1}_{\Omega} \langle \cdot \rangle_Y + \kappa \mathcal{P}_{\varepsilon} \mathcal{L}_{\omega,y}^{-1} \big) u=\lambda u,
	\end{equation*}
	which yields that $\kappa \mathcal{P}_{\varepsilon} \mathcal{L}^{-1}_{\omega,y} u = \lambda u$.
	Apply the projection $\mathcal{P}_{\varepsilon}$ on both sides and note that $\mathcal{P}_{\varepsilon}$ commutates with $\mathcal{L}_{\omega,y}^{-1}$, we get that 
	\begin{equation}\label{formula5.3}
		\kappa \mathcal{L}^{-1}_{\omega,y} \mathcal{P}_{\varepsilon} u = \lambda \mathcal{P}_{\varepsilon} u .
	\end{equation}
We write $\mathcal{P}_{\varepsilon} u$ in the form of
	\begin{equation}\label{form_u}
		\mathcal{P}_{\varepsilon} u(x,y)= \sum_{\mathbf{n} \in \widehat{\Pi}_{\varepsilon}} \mathbbm{1}_{Y^{\mathbf{n}}_{\varepsilon}}(x) u^{\mathbf{n}}_{\varepsilon}(y), \qquad\mathrm{where}  \ \langle u^{\mathbf{n}}_{\varepsilon}\rangle_Y =0,
	\end{equation}
	then \eqref{formula5.3} implies that 
	\begin{equation*}
		-\mathcal{L}_{\omega,y}^{-1}  u^{\mathbf{n}}_{\varepsilon} =\kappa^{-1} \lambda u^{\mathbf{n}}_{\varepsilon},\quad \langle u^{\mathbf{n}}_{\varepsilon}\rangle_Y =0, \quad \mathrm{for}\ \mathrm{any}\ \varepsilon>0 \ \mathrm{and} \ \mathbf{n}\in \widehat{\Pi}_{\varepsilon}.
	\end{equation*}
	Therefore, there exists $i \in \mathbb{N}$ such that $ \lambda = \kappa \alpha_i$. This completes the proof.
\hfill $\square$

Then we have the following proposition.

\begin{proposition}\label{sigma_2_charac}
	$\lambda\in \mathbb{C}$ is in the residual spectrum if and only if $\gamma_{\kappa}(\lambda)^{-1} $ is the (matrix-valued) eigenvalue of $\widehat{\mathcal{L}}_{\delta}^{-1} +B_{\varepsilon,\delta,\lambda}$, where
	$B_{\varepsilon,\delta,\lambda}$ is defined by
	\begin{equation*}
		B_{\varepsilon,\delta,\lambda} := \big(  I_m - \lambda^{-1}\gamma_{\kappa}(\lambda)^{-1}  \big) \mathbbm{1}_{\Omega} (\mathcal{P}_{\varepsilon} - \mathrm{Id}) \widehat{\mathcal{L}}_{\delta}^{-1} .
	\end{equation*}
	Moreover, the multiplicity of $\lambda$ is same as the multiplicity of $\gamma_{\kappa}(\lambda)^{-1} $.
\end{proposition}
\begin{proof}
	Let $\lambda \in \mathbb{C}$ is in the residual spectrum, by definition, there exists an eigenfunction $u$ such that $\int_Y u(x,y)\,dy \neq 0$ and
	\begin{equation}\label{eigen_equation_5}
	\big( \widehat{\mathcal{L}}_{\delta}^{-1}  \mathbbm{1}_{\Omega}\langle \cdot \rangle_Y + \kappa \mathcal{P}_{\varepsilon} \mathcal{L}_{\omega,y}^{-1} \big) u=\lambda u,
\end{equation}
	apply $\mathcal{P}_{\varepsilon}$ on both sides of above, and note that $\mathcal{P}_{\varepsilon}$ commutates with $\mathcal{L}_{\omega,y}^{-1}$,  we get
	\begin{equation*}
		  \kappa  \mathcal{L}_{\omega,y}^{-1}\mathcal{P}_{\varepsilon}  u - \lambda \mathcal{P}_{\varepsilon}u = -\mathcal{P}_{\varepsilon}  \widehat{\mathcal{L}}_{\delta}^{-1}  \mathbbm{1}_{\Omega} \langle u\rangle_Y ,
	\end{equation*}
	fixed $x$, we though the above equation only depends on $y$, then we solve it as
	\begin{equation}\label{rep_W_u}
		\mathcal{P}_{\varepsilon}u(x,y) = b_{\kappa,\lambda}(y) 	\mathcal{P}_{\varepsilon}  \widehat{\mathcal{L}}_{\delta}^{-1}  \mathbbm{1}_{\Omega} \langle u\rangle_Y (x),
	\end{equation}
where $b_{\kappa,\lambda}(y) = - (\kappa  \mathcal{L}_{\omega,y}^{-1} - \lambda   )^{-1} [I_m](y)$. Now substitute \eqref{rep_W_u} to \eqref{eigen_equation_5} we obtain
	\begin{equation}\label{int_1}
		\widehat{\mathcal{L}}_{\delta}^{-1} \mathbbm{1}_{\Omega} \langle u\rangle_Y(x) + (\lambda b_{\kappa,\lambda}(y) -I_m)	\mathcal{P}_{\varepsilon}\widehat{\mathcal{L}}_{\delta}^{-1} \mathbbm{1}_{\Omega} \langle u\rangle_Y (x) = \lambda u(x,y),
	\end{equation}
	then integrate \eqref{int_1} with respect to $y$ in $Y$, and apply $\mathbbm{1}_{\Omega}$, we get that
	\begin{equation*}
		\widehat{\mathcal{L}}_{\delta}^{-1}  \mathbbm{1}_{\Omega}  \langle u\rangle_Y(x) + \big(  I_m- \lambda^{-1}\gamma(\kappa,\lambda)^{-1}  \big) \mathbbm{1}_{\Omega}(\mathcal{P}_{\varepsilon} - \mathrm{Id}) \widehat{\mathcal{L}}_{\delta}^{-1} \mathbbm{1}_{\Omega} \langle u\rangle_Y (x) =\gamma_{\kappa}(\lambda)^{-1} \mathbbm{1}_{\Omega} \langle u\rangle_Y (x) .
	\end{equation*}
	This shows that $\gamma_{\kappa}(\lambda)^{-1} $ is the (matrix-valued) eigenvalue of $\widehat{\mathcal{L}}_{\delta}^{-1} +B_{\varepsilon,\delta,\lambda}$.

	We assume that $u_1,\cdots,u_n$ are independent eigenfunctions of $\big( \widehat{\mathcal{L}}_{\delta}^{-1}  \mathbbm{1}_{\Omega}\langle \cdot \rangle_Y + \kappa \mathcal{P}_{\varepsilon} \mathcal{L}_{\omega,y}^{-1} \big)$ for eigenvalue $\lambda$ and $a_1 \mathbbm{1}_{\Omega}\langle u_1 \rangle_Y + \cdots a_n  \mathbbm{1}_{\Omega} \langle u_n \rangle_Y=0$, then by \eqref{int_1} we obtain that $a_1 u_1 + \cdots a_n u_n=0$, hence $a_i=0$. This shows that the multiplicity of $\lambda$ is not larger than that of $\gamma(\kappa,\lambda)^{-1} $.
	
	Conversely, suppose that $\gamma(\kappa,\lambda)^{-1} $ is an eigenvalue of $\widehat{\mathcal{L}}_{\delta}^{-1}  +B_{\varepsilon,\delta,\lambda}$, there exists a nonzero $f \in L^2(\Omega)$ such that
	\begin{equation*}
		\widehat{\mathcal{L}}_{\delta}^{-1} f +B_{\varepsilon,\delta,\lambda} f = \gamma(\kappa,\lambda)^{-1}  f,
	\end{equation*}
	we define 
	\begin{equation*}
		u(x,y):=\lambda^{-1}  \widehat{\mathcal{L}}_{\delta}^{-1} f (x) + \big( b_{\kappa, \lambda}(y)-\lambda^{-1} I_m \big)  \mathcal{P}_{\varepsilon}  \widehat{\mathcal{L}}_{\delta}^{-1} f (x),
	\end{equation*}
	then $\mathbbm{1}_{\Omega} \langle u\rangle_Y =f$, in particular, $\langle u \rangle_Y \neq 0$. Moreover, 
	\begin{equation*}
		\begin{aligned}
		&\widehat{\mathcal{L}}_{\delta}^{-1}\mathbbm{1}_{\Omega}  \langle u \rangle_Y - \kappa  \mathcal{P}_{\varepsilon} \mathcal{L}_{\omega,y}^{-1} u \\
		& = \widehat{\mathcal{L}}_{\delta}^{-1}  f (x)- \kappa\lambda^{-1} \mathcal{L}_{\omega,y}^{-1}[I_m](y) \mathcal{P}_{\varepsilon}	\widehat{\mathcal{L}}_{\delta}^{-1}   f(x)  -\kappa \big( \mathcal{L}_{\omega}^{-1}b_{\kappa,\lambda}(y)-\lambda^{-1}\mathcal{L}_{\omega}^{-1}[I_m](y) \big)  \mathcal{P}_{\varepsilon}	\widehat{\mathcal{L}}_{\delta}^{-1}  f (x)\\
			& =   \widehat{\mathcal{L}}_{\delta}^{-1}  f (x) + \big(\lambda b_{\kappa,\lambda}(y) -I_m \big) \mathcal{P}_{\varepsilon}\widehat{\mathcal{L}}_{\delta}^{-1}  f (x)  \\
			& = \lambda u,
		\end{aligned}
	\end{equation*}
	which shows that $\lambda $ is in the residual spectrum. Moreover, if $f_1,\cdots,f_n$ are independent eigenfunctions, then of course $ u_1,\cdots,u_n$ are independent since $f =\mathbbm{1}_{\Omega}   \langle u \rangle_Y$, so the multiplicity of $\lambda$ is not smaller than the multiplicity of $\gamma_{\kappa}(\lambda)^{-1}$. The proof is complete.
\end{proof}

\noindent \textit{Proof of Theorem \ref{thm:residual spectrum}}. We define an inner product on $H^{-1}(\Omega)$
by
\begin{equation}
	\langle u,v \rangle_{\delta} := \langle  u, \widehat{\mathcal{L}}_{\delta}^{-1} v \rangle_{H^{-1}(\Omega),H_0^1(\Omega)},
\end{equation}
for $u,v \in H^{-1}(\Omega)$. We show that $\| \cdot \|_{\delta}$ is equivalent to $\|\cdot \|_{H^{-1}(\Omega)}$. For any $u \in H^{-1}(\Omega)$, since $\widehat{\mathcal{L}}_{\delta}^{-1}$ is a bijection from $H^{-1}(\Omega)$ to $H_0^1(\Omega)$, we have
\begin{equation*}
	\| u \|_{H^{-1}(\Omega)} \sim \| \widehat{\mathcal{L}}_{\delta}^{-1}  u\|_{H_0^1(\Omega)}.
\end{equation*}
It then follows from the definition of $\| \cdot \|_{\delta}$ that
\begin{equation}\label{norm equiv 1}
	\| u \|_{\delta}^2 \leq \| u \|_{H^{-1}(\Omega)} \| \widehat{\mathcal{L}}_{\delta}^{-1}  u \|_{H_0^1(\Omega)} \leq C\| u \|_{H^{-1}(\Omega)}^2,
\end{equation}
which yields that
\begin{equation}\label{norm equiv 2}
	\| u \|_{H^{-1}(\Omega) } = \sup_{0\neq v \in H_0^1(\Omega)} \frac{\langle u,v \rangle_{H^{-1}(\Omega),H_0^1(\Omega)}}{\| v\|_{H^1(\Omega)}}  =  \sup_{0\neq w \in H^{-1}(\Omega)} \frac{\langle u,  \widehat{\mathcal{L}}_{\delta}^{-1} w \rangle_{H^{-1}(\Omega),H_0^1(\Omega)}}{\| w\|_{H^{-1}(\Omega)}} \leq \| u \|_{\delta}.
\end{equation}
\eqref{norm equiv 1} and \eqref{norm equiv 2} imply that $\| \cdot \|_{\delta}$ is equivalent to $\|\cdot \|_{H^{-1}(\Omega)}$. 

Since $\| \cdot \|_{\delta}$ is equivalent to $\|\cdot \|_{H^{-1}(\Omega)}$ and $\widehat{\mathcal{L}}_{\delta}^{-1}$ is compact on $H^{-1}(\Omega)$, we obtain that $\widehat{\mathcal{L}}_{\delta}^{-1} +B_{\varepsilon,\delta,\lambda} :(H^{-1}(\Omega), \| \cdot \|_{\delta}) \rightarrow (H^{-1}(\Omega), \| \cdot \|_{\delta})$ is compact. The self-adjointness follows from
\begin{equation*}
	\langle \widehat{\mathcal{L}}_{\delta}^{-1} u ,v \rangle_{\delta } = \langle \widehat{\mathcal{L}}_{\delta}^{-1}u , \widehat{\mathcal{L}}_{\delta}^{-1} v \rangle_{H^{-1}(\Omega),H_0^1(\Omega)} = \langle u ,\widehat{\mathcal{L}}_{\delta}^{-1} v \rangle_{\delta}
\end{equation*}
and
\begin{equation*}
	\langle \mathbbm{1}_{\Omega}\mathcal{P}_{\varepsilon}\widehat{\mathcal{L}}_{\delta}^{-1}u ,v \rangle_{\delta} = \langle  \mathcal{P}_{\varepsilon}\widehat{\mathcal{L}}_{\delta}^{-1} u , \widehat{\mathcal{L}}_{\delta}^{-1} v \rangle_{H^{-1}(\Omega),H_0^1(\Omega)} =\langle  \widehat{\mathcal{L}}_{\delta}^{-1}u , \mathcal{P}_{\varepsilon}\widehat{\mathcal{L}}_{\delta}^{-1} v \rangle_{H^{-1}(\Omega),H_0^1(\Omega)} = \langle u , \mathbbm{1}_{\Omega}\mathcal{P}_{\varepsilon}\widehat{\mathcal{L}}_{\delta}^{-1} v \rangle_{\delta}.
\end{equation*}
Therefore, $\widehat{\mathcal{L}}_{\delta}^{-1} +B_{\varepsilon,\delta,\lambda} :(H^{-1}(\Omega), \| \cdot \|_{\delta}) \rightarrow (H^{-1}(\Omega), \| \cdot \|_{\delta})$ is a compact self-adjoint operator on $(H^{-1}(\Omega), \| \cdot \|_{\delta})$. The conclusion of Theorem \ref{thm:residual spectrum} immediately follows from the standard stability theorem for compact self-adjoint operators and the following estimates:
\begin{equation*}
	\begin{aligned}
		\| B_{\varepsilon,\delta,\lambda} \,f \|_{\delta} & \leq C \| B_{\varepsilon,\delta,\lambda}\,f \|_{L^2(\Omega)} \\
		&\leq C \big( 1 + |\lambda|^{-1 }  \| \gamma_{\kappa} (\lambda)^{-1} \|  \big)\| \mathcal{P}_{\varepsilon} - \mathrm{Id}_{L^2(\widehat{\Omega}_{\varepsilon} )} \|_{H^1(\Omega ) \rightarrow L^2(\widehat{\Omega}_{\varepsilon}  )} \| \widehat{\mathcal{L}}_{\delta}^{-1} f \|_{H^1(\Omega)} \\
		& \leq C \varepsilon \big( 1 + |\lambda|^{-1 }  \| \gamma_{\kappa} (\lambda)^{-1} \|  \big) \| f \|_{\delta}.
	\end{aligned}
\end{equation*}
We are done.
\hfill $\square$

\section{The scalar case of $m=1$}\label{sec:scalar case}

We assume that $\mathcal{L}_{\omega,y}^{-1} \, \psi_i = \beta_i \psi_i$, where the associated normalized eigenfunction $\psi_i$ has nonzero mean.
Then
\begin{equation}
	(\kappa  \mathcal{L}_{\omega,y}^{-1} - \lambda   )^{-1} [I_m](y)  = \left\{
	\begin{aligned}
		& \sum_i \frac{ \int_Y \psi_i (y)\,dy}{\kappa \beta_i -\lambda}\otimes \psi_i 
		& \mathrm{in}\ \omega, \\
		& -\lambda^{-1} I_m & \mathrm{in}\ Y\setminus \omega.
	\end{aligned}\right.
\end{equation}
The $pq$-th element of $\gamma_{\kappa}(\lambda) $ is 
\begin{equation*}
	( \gamma_{\kappa}(\lambda) )_{pq}= \sum_{i \geq 1} \frac{ (\int_Y \psi_i^p (y)\,dy)(\int_Y \psi_i^q (y)\,dy) }{\lambda - \kappa \beta_i }  + \frac{1-\theta}{\lambda} \delta_{pq} .
\end{equation*}
We denote $\beta_{\kappa} (\lambda) = \gamma_{\kappa}(\lambda^{-1})$. From now, we assume that $m=1$, then
\begin{equation}
	\beta_{\kappa} (\lambda) =\lambda \sum_{i \geq 1} \frac{ (\int_Y \psi_i (y)\,dy)^2}{1 - \kappa \beta_i \lambda}  + (1-\theta) \lambda.
\end{equation} 
then $\beta_{\kappa}(\lambda)$ is increasing in each interval $(\beta_i ,\beta_{i+1})$, and $\beta_{\kappa}'(\lambda)$ has a lower bound $1- \theta$, which is easily seen by a simple computation
\begin{equation*}
	\beta_{\kappa}'(\lambda) = \sum_{k=1}^{\infty} \frac{(\int_Y \psi_i (y)\,dy)^2 }{(1 - \kappa \beta_i \lambda)^2} +1-\theta\geq 1-\theta>0.
\end{equation*}

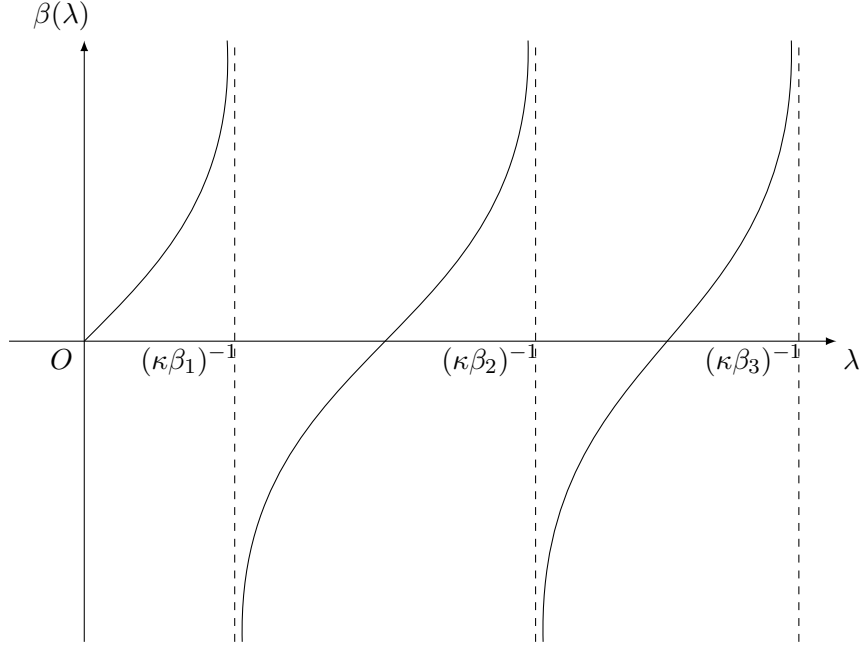
\begin{figure}[h]
	\begin{tikzpicture}
		\draw[-latex] (-1,0) -- (10,0);
		\draw[-latex] (0,-4) -- (0,4);
		\coordinate[label=$O$] (O) at (-0.3,-0.5);
		\coordinate[label=$\lambda$] (lambda) at (10.2,-0.5);
		\coordinate[label=$\beta(\lambda)$] (beta) at (-0.3,4);
		\draw[dashed] (2,-4) -- (2,4); 
		\draw[dashed] (6,-4) -- (6,4); 
		\draw[dashed] (9.5,-4) -- (9.5,4); 
		\draw (0,0) .. controls (1,1) and (2,2) .. (1.9,4); 
		\draw (2.1,-4) .. controls (2,0) and (6,0) .. (5.9,4);
		\draw (6.1,-4) .. controls (6,0) and (9.5,0) .. (9.4,4);
		\coordinate[label=$(\kappa \beta_1)^{-1}$] (omega1) at (1.4,-0.6);
		\coordinate[label=$(\kappa \beta_2)^{-1}$] (omega2) at (5.4,-0.6);
		\coordinate[label=$(\kappa \beta_3)^{-1}$] (omega3) at (8.9,-0.6);
	\end{tikzpicture}
	\caption{The graph of $\beta(\lambda)$.}
\end{figure}

\begin{theorem}\label{thm:scalar}
	For any $i\geq 0$ and $j \geq 1$, let $\lambda_{i,j}$ be the unique solution of $\beta_{\kappa}(\lambda) = \theta_j$ in the interval $\big( (\kappa\beta_i)^{-1} , (\kappa\beta_{i+1})^{-1} \big)$, where $\theta_j$ is the $j$-th eigenvalue of $
	\widehat{\mathcal{L}}_{\delta}$ in the increasing order. Let $\lambda_{i,j}^{\varepsilon}$ be the unique solution of $\beta_{\kappa}(\lambda) =  \theta_j^{\varepsilon}$ in the interval $\big( (\kappa\beta_i)^{-1} , (\kappa\beta_{i+1})^{-1} \big)$, where $\theta_j^{\varepsilon}$ is the $j$-th eigenvalue of $(\widehat{\mathcal{L}}_{\delta}^{-1} +B_{\varepsilon,\delta,\lambda})^{-1} $.
	There exists $C>0$, depends only in $\Omega$, such that if 
	\begin{equation}\label{constchose}
		\varepsilon < \frac{C}{\theta_j},
	\end{equation}
	 then
	\begin{equation}
			|\lambda_{i,j}  -\lambda_{i,j}^{\varepsilon}| \leq C \varepsilon \theta_j \big(\theta_j+(\kappa\beta_{i+1})^{-1} \big) .
	\end{equation}
\end{theorem}

\begin{proof}
	By Theorem \ref{thm:residual spectrum} for $m=1$, there exists $M>1$, depends only on $\Omega$, such that
	\begin{equation}\label{eigen_esti_5}
		| \beta_{\kappa}(\lambda_{i,j})^{-1} - \beta_{\kappa}(\lambda^{\varepsilon}_{i,j})^{-1}  | \leq M \varepsilon \big| 1-\lambda_{i,j}^{\varepsilon}\, \beta_{\kappa}(\lambda^{\varepsilon}_{i,j})^{-1}   \big|,
	\end{equation} 
	then since $\beta_{\kappa}$ is smooth in each interval $\big((\kappa \beta_i)^{-1},(\kappa \beta_{i+1})^{-1} \big)$, we have
	\begin{equation}\label{eigen_estimate}
		\begin{aligned}
			|\lambda_{i,j}  - \lambda_{i,j}^{\varepsilon}|& \leq \frac{1}{ \inf_{ \lambda} \beta_{\kappa}'(\lambda)} |\beta_{\kappa}(\lambda_{i,j})  - \beta_{\kappa}(\lambda_{i,j}^{\varepsilon} )| \\
			& \leq \frac{|\beta(\lambda_{i,j})  \beta(\lambda^{\varepsilon}_{i,j})  |}{1- \theta } 	\left|  \frac{1}{\beta(\lambda_{i,j})} - \frac{1}{\beta(\lambda^{\varepsilon}_{i,j}) } \right| \\
			&\leq \frac{M \varepsilon}{1-\theta} \beta(\lambda_{i,j})  |\beta(\lambda_{i,j}^{\varepsilon}) - \lambda_{i,j}^{\varepsilon}|.
		\end{aligned}
	\end{equation}
	Then we have three cases: 
	\begin{itemize}
		\item[(a)] If $|\beta(\lambda_{i,j}^{\varepsilon})| \leq 2M\lambda_{i,j}^{\varepsilon}$, then \eqref{eigen_estimate} gives that
		\begin{equation}\label{last0}
			|\lambda_{i,j}  - \lambda_{i,j}^{\varepsilon}| \leq C\varepsilon \beta(\lambda_{i,j}) \lambda_{i,j}^{\varepsilon} \leq C\varepsilon\theta_j (\kappa\beta_{i+1})^{-1} .
		\end{equation}
		\item[(b)] If $\beta(\lambda_{i,j}^{\varepsilon}) > 2 M\lambda_{i,j}^{\varepsilon}$, note that $M>1$, then \eqref{eigen_esti_5} implies that 
		\begin{equation*}
			\frac{1}{\beta(\lambda_{i,j})} - \frac{1}{\beta(\lambda_{i,j}^{\varepsilon})} \leq M \varepsilon^{\frac{1}{2}} \Big(  \frac{\beta(\lambda_{i,j}^{\varepsilon}) - \lambda_{i,j}^{\varepsilon}}{\beta(\lambda_{i,j}^{\varepsilon})}   \Big),
		\end{equation*} 
		which yields that
		\begin{equation}\label{eiti}
			\big(\beta(\lambda_{i,j}^{\varepsilon} )   -   \lambda_{i,j}^{\varepsilon} \big)  \big(1-M\varepsilon^{\frac{1}{2}} \beta(\lambda_{i,j})   \big) \leq \beta(\lambda_{i,j})    - \lambda_{i,j}^{\varepsilon} ,
		\end{equation}
		we choose $C$ in \eqref{constchose} by letting 
		\begin{equation*}
			C=\frac{1}{2M},
		\end{equation*}
		that $ 2M\varepsilon^{\frac{1}{2}} \beta(\lambda_{i,j})  <1$, then \eqref{eiti} implies that
		\begin{equation*}
			\beta(\lambda_{i,j}^{\varepsilon}) -\lambda_{i,j}^{\varepsilon} \leq 2 \big(\beta(\lambda_{i,j}) - \lambda_{i,j}^{\varepsilon} \big),
		\end{equation*}
		and \eqref{eigen_estimate} gives that
		\begin{equation}\label{last1}
			|\lambda_{i,j}  - \lambda_{i,j}^{\varepsilon}| \leq C \varepsilon^{\frac{1}{2}} \beta(\lambda_{i,j}) \big(\beta(\lambda_{i,j}) - \lambda_{i,j}^{\varepsilon} \big)  \leq C \varepsilon \beta^2(\lambda_{i,j})= C\varepsilon\theta^2_j   .
		\end{equation}
		\item[(c)] If $\beta(\lambda_{i,j}^{\varepsilon}) <-2 M\lambda_{i,j}^{\varepsilon}$, then \eqref{eigen_esti_5} implies that 
		\begin{equation*}
			\frac{1}{\beta(\lambda_{i,j})} + \frac{1}{|\beta(\lambda_{i,j}^{\varepsilon})|} \leq M \varepsilon \Big(  \frac{|\beta(\lambda_{i,j}^{\varepsilon})| + \lambda_{i,j}^{\varepsilon}}{|\beta(\lambda_{i,j}^{\varepsilon})|}   \Big),
		\end{equation*} 
		which yields that
		\begin{equation}
			\big( |\beta(\lambda_{i,j}^{\varepsilon} )  | + \lambda_{i,j}^{\varepsilon} \big)  \big(1-M\varepsilon^{\frac{1}{2}} \beta(\lambda_{i,j})   \big) \leq    \lambda_{i,j}^{\varepsilon} - \beta(\lambda_{i,j})   ,
		\end{equation}
		we still choose $C$ in \eqref{constchose} by letting 
		\begin{equation*}
			C=\frac{1}{2M},
		\end{equation*}
		then
		\begin{equation*}
			|\beta(\lambda_{i,j}^{\varepsilon}) |+ \lambda_{i,j}^{\varepsilon} \leq 2 \big(\lambda_{i,j}^{\varepsilon} - \beta(\lambda_{i,j}) \big),
		\end{equation*}
		and \eqref{eigen_estimate} gives that
		\begin{equation}
			|\lambda_{i,j}  - \lambda_{i,j}^{\varepsilon}| \leq C \varepsilon^{\frac{1}{2}} \beta(\lambda_{i,j})\big(\lambda_{i,j}^{\varepsilon} - \beta(\lambda_{i,j}) \big)  \leq C\varepsilon^{\frac{1}{2}} \theta_j (\kappa\beta_{i+1})^{-1} .
		\end{equation}
	\end{itemize}
	In summary, if the constant $C= (2M)^{-1}$ in \eqref{constchose}, we have
	\begin{equation*}
		|\lambda_{i,j}  -\lambda_{i,j}^{\varepsilon}| \leq C \varepsilon \theta_j (\theta_j+(\kappa\beta_{i+1})^{-1} ),
	\end{equation*}
	the proof is complete.
\end{proof}

\bibliographystyle{siam}
\bibliography{mybib}

\end{document}